\documentclass[11pt,reqno]{amsart}

{\theoremstyle{plain}
 \newtheorem{theorem}{Theorem}

\newtheorem{lemma}{Lemma}
}

\DeclareMathOperator{\re}{Re}

\begin{document}

\title[The constant factor in the asymptotic for practical numbers]{The constant factor in the asymptotic \\ for practical numbers}

\begin{abstract}
An integer $n\ge 1$ is said to be practical if every natural number $m \le n$ can be expressed as a sum of distinct positive divisors of $n$.
The number of practical numbers up to $x$ is asymptotic to $c x/\log x$, where $c$ is a constant. In this note we show that $c=1.33607...$. 
\end{abstract}

\author{Andreas Weingartner}
\address{ 
Department of Mathematics,
351 West University Boulevard,
 Southern Utah University,
Cedar City, Utah 84720, USA}
\email{weingartner@suu.edu}
\maketitle

\section{Introduction}

Following Srinivasan \cite{Sri}, we call an integer $n\ge 1$ practical if every natural number $ m \le n$ can be expressed as a sum of distinct positive divisors of $n$. Let $P(x)$ be the number of practical numbers up to $x$.
Margenstern \cite{Mar} conjectured that $P(x)$ is asymptotic to $c x/\log x$ and gave the empirical estimate $c\approx 1.341$.
This conjecture was confirmed with the estimate \cite[Thm. 1.1]{PDD}
\begin{equation}\label{Pasymp}
P(x)=\frac{c x}{\log x} \left\{1+O\left( \frac{\log \log x}{\log x}  \right)\right\}.
\end{equation}
The constant factor $c$ is given by the sum of an infinite series \cite[Thm. 1]{CFAE},
\begin{equation}\label{pracon}
c= \frac{1}{1-e^{-\gamma}} \sum_{n\in \mathcal{P}} \frac{1}{n}  \Biggl( \sum_{p\le \sigma(n)+1}\frac{\log p}{p-1} - \log n\Biggr) \prod_{p\le \sigma(n)+1} \left(1-\frac{1}{p}\right),
\end{equation}
where $\gamma$ is Euler's constant, $\mathcal{P}$ is the set of practical numbers, $p$ runs over primes and $\sigma(n)$ is the sum of the positive divisors of $n$.
As a consequence, Corollary 1 of \cite{CFAE} states that  $1.311 < c <1.693$. 
The purpose of this note is to establish a more precise estimate for $c$. 
\begin{theorem}\label{thm1}
We have $1.336073<c<1.336077$.
\end{theorem}
While in \cite{CFAE} we  used the extremal behavior of $\sigma(n)$ to estimate the contribution to \eqref{pracon} from large $n$,
here we apply the new identity in Lemma \ref{lem2} together with the multiplicativity of $\sigma(n)$. 
As a result, the remaining gap in Theorem \ref{thm1} is almost entirely due to the error term of Lemma \ref{pslem} when estimating the inner sum over primes in  \eqref{pracon}.

\section{Lemmas}
As Lemmas \ref{lem1} and \ref{lem2} apply to other sets of numbers besides the practical numbers, we recall the general setup from 
\cite{SPA}.  

Let $\theta$ be an arithmetic function, $\theta: \mathbb{N}\to \mathbb{R}\cup \{\infty\}$.  We write $\mathcal{B}$  to denote the set of positive integers containing $n=1$ and all those $n \ge 2$ with prime factorization  $n=p_1^{\alpha_1} \cdots p_k^{\alpha_k}$, \
$p_1< p_2< \ldots < p_k$, which satisfy 
\begin{equation*}\label{Bdef}
p_{j+1} \le \theta (p_1^{\alpha_1}\cdots p_{j}^{\alpha_{j}} ) \qquad (0\le j < k),
\end{equation*}
where $p_1^{\alpha_1}\cdots p_{j}^{\alpha_{j}}$ 
is understood to be $1$ when $j=0$. 
As in \cite{SPA}, we assume that
\begin{equation}\label{theta}
\theta: \mathbb{N}\to \mathbb{R}\cup \{\infty\}, \quad \theta(1)\ge 2, \quad  \theta(n)\ge P^+(n) \quad (n\ge 2),
\end{equation}
where $P^+(n)$ denotes the largest prime factor of $n$. 
The assumptions in \eqref{theta} only eliminate the trivial case $\mathcal{B}=\{1\}$. 
Let $B(x)$ be the number of positive integers $n\le x$ in $\mathcal{B}$. We write $\chi(n)$ to denote the characteristic function of the set $\mathcal{B}$. 

Sierpinski \cite{Sier} and Stewart \cite{Stew} found that if $\theta(n)=\sigma(n)+1$, then $\mathcal{B}=\mathcal{P}$.

\begin{lemma}\label{lem1}
Let $\theta$ satisfy \eqref{theta}. For $\re(s)>1$ we have 
\begin{equation*}
 \sum_{n\ge 1} \frac{\chi(n)}{n^s} \prod_{p\le \theta(n)} \left( 1-\frac{1}{p^s}\right)=1.
\end{equation*}
If $B(x)=o(x)$, the equation also holds at $s=1$. 
\end{lemma}

\begin{proof}
The case $\re(s)>1$ is \cite[Lemma 1]{CFAE}. The case $s=1$ is \cite[Theorem 1]{SPA}.
When $s=1$, the series actually converges for any $\mathcal{B}$ to a value $L\in [0,1]$ (see \cite[Lemma 4]{SPA}), with
$L=1$ if and only if $B(x)=o(x)$ (see \cite[Theorem 1]{SPA}).
\end{proof}

\begin{lemma}\label{lem2}
Let $\theta$ satisfy \eqref{theta}. Let $q$ be prime and $h\in \mathbb{N}$. For $\re(s)>1$ we have
\begin{equation*}
\sum_{n\ge 1 \atop q^h ||n} \frac{\chi(n)}{n^s} \prod_{p\le \theta(n)} \left( 1-\frac{1}{p^s}\right)
= \frac{1-1/q^s}{q^{s h}}\sum_{n\ge 1 \atop \theta(n)\ge q} \frac{\chi(n)}{n^s}  \prod_{p\le \theta(n)} \left( 1-\frac{1}{p^s}\right) .
\end{equation*}
If $B(x)=o(x)$, the equation also holds at $s=1$. 
\end{lemma}

\begin{proof}
We first assume $\re(s)>1$.
 Each natural number $m$ with prime factorization $m=p_1^{\alpha_1}p_2^{\alpha_2} \cdots p_k^{\alpha_k}$, $p_1<p_2<\ldots < p_k$, factors uniquely as $m=nr$, where 
 $n =p_1^{\alpha_1}p_2^{\alpha_2}\cdots p_j^{\alpha_j} \in \mathcal{B}$ and $p_{j+1}>\theta(n)$ for some $j$ with $0\le j <k$, 
 or $m=n \in \mathcal{B}$ and $r=1$.  
 If $q^h || m$ then either $q^h || n$ or $q^h || r$. Accordingly
 $$ \sum_{m\ge 1 \atop q^h || m} \frac{1}{m^s}
  = \sum_{n\ge 1 \atop q^h || n} \frac{\chi(n)}{n^s} \prod_{p> \theta(n)} \left( 1-\frac{1}{p^s}\right)^{-1}
  + \sum_{n\ge 1 \atop \theta(n)<q} \frac{\chi(n)}{n^s} \frac{1}{q^{hs}} \prod_{p> \theta(n) \atop p\neq q} \left( 1-\frac{1}{p^s}\right)^{-1} .$$
On the other hand,
$$  \sum_{m\ge 1 \atop q^h || m} \frac{1}{m^s} 
= \frac{1}{q^{hs}}  \prod_{p\neq q} \left( 1-\frac{1}{p^s}\right)^{-1}
= \frac{1-1/q^s}{q^{hs}}  \prod_{p\ge 2} \left( 1-\frac{1}{p^s}\right)^{-1}. $$
  Setting the right-hand sides equal and dividing by $\zeta(s)= \prod_{p\ge 2} \left(1-p^{-s}\right)^{-1}$ yields
$$
 \frac{1-1/q^s}{q^{hs}} =  \sum_{n\ge 1 \atop q^h || n} \frac{\chi(n)}{n^s} \prod_{p\le \theta(n)} \left( 1-\frac{1}{p^s}\right)
 + \frac{1-1/q^s}{q^{hs}} \sum_{n\ge 1 \atop \theta(n)<q} \frac{\chi(n)}{n^s}  \prod_{p\le \theta(n) } \left( 1-\frac{1}{p^s}\right).
$$
For $\re(s)>1$, the result follows now from Lemma \ref{lem1}. 
To show that the last equation, and hence Lemma \ref{lem2}, also holds at $s=1$ if $B(x)=o(x)$, it suffices to show that 
the last two sums are continuous from the right at $s=1$. 
We have
$$  \prod_{p\le \theta(n)} \left( 1-\frac{1}{p^s}\right) \asymp s-1 +\frac{1}{\log \theta(n)},$$
uniformly for $n\ge 1$ and $1\le s \le 2$, by \cite[Eq. (24)]{CFAE}.
Thus, uniformly for $1\le s \le 2$, 
\begin{multline*}
 \sum_{n>N} \frac{\chi(n)}{n^s} \prod_{p\le \theta(n)} \left( 1-\frac{1}{p^s}\right) 
\ll \sum_{n>N} \frac{\chi(n)}{n^s}(s-1)+  \sum_{n>N} \frac{\chi(n)}{n^s}\frac{1}{\log \theta(n)} \\
=G_N(s)+H_N(s),
\end{multline*}
say. Since $B(x)=o(x)$, partial summation shows that $G_N(s)=o(1)$ as $N\to \infty$, uniformly for $1\le s \le 2$.
We have $H_N(s)\le H_N(1)=o(1)$ as $N\to \infty$, since the series in Lemma \ref{lem1} converges when $s=1$.  
It follows that $G_N(s)+H_N(s)=o(1)$ as $N\to \infty$, uniformly for $1\le s \le 2$, which concludes the proof. 
\end{proof}

\begin{lemma}\label{lem3}
Let $\theta(n)=\sigma(n)+1$. We have 
$$  \sum_{n \ge 1} \frac{\chi(n)}{n}\log\left(\frac{\sigma(n)}{n}\right) \prod_{p\le \theta(n)} \left(1-\frac{1}{p}\right) 
=\sum_{n \ge 1} \frac{\chi(n)}{n} \sum_{q\le \theta(n)} W_q \prod_{p\le \theta(n)} \left(1-\frac{1}{p}\right),
$$
where $p$ and $q$ run over primes and 
$$0< W_q := \sum_{h\ge 1}  \frac{1-1/q}{q^{h}} \log\left(\frac{1-1/q^{h+1}}{1-1/q}\right)<\frac{1}{q(q-1)}.
$$
\end{lemma}

\begin{proof}
The multiplicativity of $\sigma(n)$ and Lemma \ref{lem2} yield
\begin{equation*}
\begin{split}
&  \sum_{n \ge 1} \frac{\chi(n)}{n}\log\left(\frac{\sigma(n)}{n}\right) \prod_{p\le \theta(n)} \left(1-\frac{1}{p}\right) \\
 = &   \sum_{n \ge 1} \frac{\chi(n)}{n} \sum_{q^h || n} \log\left(1+\frac{1}{q}+\cdots +\frac{1}{q^h}\right) \prod_{p\le \theta(n)} \left(1-\frac{1}{p}\right) \\
=  & \sum_{q\ge 2 \atop h\ge 1} \log\left(\frac{1-1/q^{h+1}}{1-1/q}\right)
\sum_{n\ge 1 \atop q^h ||n} \frac{\chi(n)}{n} \prod_{p\le \theta(n)} \left( 1-\frac{1}{p}\right) \\
= &  \sum_{q\ge 2 \atop h\ge 1} \log\left(\frac{1-1/q^{h+1}}{1-1/q}\right)
 \frac{1-1/q}{q^{h}}\sum_{n\ge 1 \atop \theta(n)\ge q} \frac{\chi(n)}{n}  \prod_{p\le \theta(n)} \left( 1-\frac{1}{p}\right) \\
= & \sum_{q\ge 2} W_q \sum_{n\ge 1 \atop \theta(n)\ge q} \frac{\chi(n)}{n}  \prod_{p\le \theta(n)} \left( 1-\frac{1}{p}\right)\\
= & \sum_{n \ge 1} \frac{\chi(n)}{n} \sum_{q\le \theta(n)} W_q \prod_{p\le \theta(n)} \left(1-\frac{1}{p}\right).
\end{split}
\end{equation*}
The convergence of these series follows from \eqref{Pasymp}.
We have
$$ 0<W_q < \sum_{h\ge 1}  \frac{1-1/q}{q^{h}} \log\left(\frac{1}{1-1/q}\right)=-\frac{\log(1-1/q)}{q} < \frac{1}{q(q-1)} .$$
\end{proof}

\begin{lemma}\label{pslem}
Let
\begin{equation}\label{etadef}
 \eta(x) = \sum_{p\le x} \frac{\log p}{p-1}  - \log x + \gamma 
\end{equation}
and 
\begin{equation}\label{deltadef}
 \delta(x) = \sum_{p\le x} \frac{\log p}{p}  - \log x + \gamma + \sum_{p\ge 2} \frac{\log p}{p(p-1)} = \eta(x) +   \sum_{p>x} \frac{\log p}{p(p-1)}
\end{equation}
We have 
$$ |\eta(x)| \le M_k, \qquad |\delta(x)| \le M_k \qquad (x\ge 2^k),$$
where $M_k$ is given by 
\begin{table}[h]\label{table1}
 \begin{tabular}{ | c | c | }
    \hline
    $\ k \ $ & $M_k \times 10^{5}$ \\ \hline
    $24$ & $36.80 $ \\ \hline
    $25$ &  $27.65$ \\ \hline
    $26$ & $17.60 $ \\ \hline
    $27$ & $13.04 $ \\ \hline
    $28$ &  $8.173 $  \\ \hline
     \end{tabular}
     \qquad
  \begin{tabular}{ | c | c | }
    \hline
    $\ k \ $ & $M_k \times 10^{5}$ \\ \hline
    $29$ &  $6.377 $  \\ \hline
    $30$ &  $5.122$ \\ \hline
     $31$ & $3.143 $ \\ \hline
    $32$ & $2.174 $  \\ \hline
     $33$ & $1.654$  \\ \hline
     \end{tabular}
     \qquad
      \begin{tabular}{ | c | c | }
    \hline
    $\ k \ $ & $M_k \times 10^{5}$ \\ \hline
    $34$ & $1.101$  \\ \hline
    $35$ & $0.833$  \\ \hline
    $36$ & $0.569$  \\ \hline
    $37$ & $0.438$  \\ \hline
    $38$ & $0.305$  \\ \hline
     \end{tabular}
        \medskip
     \caption{The values of $M_k$ are best possible apart from rounding. }\label{table1}
\end{table}
\end{lemma}

\begin{proof}
Let $\vartheta(x) = \sum_{p\le x} \log p$ and $\psi(x)=\sum_{p^j\le x} \log p$. 
By Eq. (4.21) of Rosser and Schoenfeld \cite{RS}, we have
\begin{equation}\label{deltadiff}
\delta(y)-\delta(x) = \frac{\vartheta(y)-y}{y}-\frac{\vartheta(x)-x}{x} + \int_x^y \frac{\vartheta(t)-t}{t^2} dt.
\end{equation}
B\"{u}the \cite[Thm. 2]{But} showed that $-1.95 \sqrt{x} \le \vartheta(x)-x <-0.05 \sqrt{x}$ for $1423\le x \le 10^{19}$. 
Together with \eqref{deltadiff} we get, for $1423\le x \le y \le 10^{19}$,
\begin{equation}\label{ineq1}
  \delta(x) + \frac{x-\vartheta(x)}{x} -\frac{3.9}{\sqrt{x}} \le \delta(y) \le  \delta(x) + \frac{x-\vartheta(x)}{x} -\frac{0.05}{\sqrt{x}}.
\end{equation}

Rosser and Schoenfeld \cite[Thm. 13]{RS} found that $|\psi(x)-\vartheta(x)|<1.4262 \sqrt{x}$ for all $x>0$. 
Dusart \cite[Prop. 3.2]{Dus} gives inequalities of the form $|\psi(x)-x|\le \epsilon_i x$ for $x\ge e^{b_i}$. 
Hence
$$ \left|\frac{\vartheta(x)-x}{x}\right| \le  \frac{1.4262}{e^{b_i/2}} + \epsilon_i \qquad (x \ge e^{b_i}),$$
where the pairs $(b_i,\epsilon_i)$ take the values $(\log (10^{19}),1.161\times 10^{-7})$, $(45,1.225\times 10^{-8})$,$\ldots$, $(500,1.215\times 10^{-11})$,
according to \cite[Table 1]{Dus}. 
Together with \eqref{deltadiff}, we find that
\begin{equation}\label{ineq2}
 |\delta(y)-\delta(x)|\le 2.630 \times 10^{-7} \qquad (10^{19}=x \le y \le e^{600}).
\end{equation}
Finally, for $y\ge e^{600}$ we use Proposition 8 from Axler \cite{Axler}:
\begin{equation}\label{ineq3}
  |\delta(y)| \le \frac{3}{40 \log^2 y} + \frac{3}{20 \log^3 y} \le 2.091 \times 10^{-7}, \qquad (y\ge e^{600}).
\end{equation}
We now show how these inequalities give rise to Table \ref{table1}. We only need to verify the case $k=38$. The other values of $M_k$ in Table \ref{table1} follow by computer calculation. With the computer we verify that 
$0<\eta(x) < \delta(x) < M_{38}=3.05\times 10^{-6}$ for $2^{38}\le x \le 2^{39}$. For $x=2^{39}$, we calculate $\delta(x)$ and $\vartheta(x)$
and use \eqref{ineq1} to obtain
$$ |\delta(y)|\le 2.652 \times 10^{-6}\qquad (2^{39} \le y \le 10^{19}).$$
With \eqref{ineq2} this implies
$$ |\delta(y)| \le 2.652 \times 10^{-6} + 0.263 \times 10^{-6} = 2.915 \times 10^{-6} \qquad (2^{39} \le y \le e^{600}).$$
Together with \eqref{ineq3} this shows that $|\delta(x)|\le M_{38}$ for $x\ge 2^{38}$. To show that $\eta(x)$ satisfies the same inequality, note that
\begin{equation}\label{deleta}
0< \delta(x)-\eta(x) =   \sum_{p>x} \frac{\log p}{p(p-1)} < \int_x^\infty \frac{\log t}{t^2} dt = \frac{1+\log x}{x} < 10^{-10}
\end{equation}
for $x\ge 2^{39}$. 
\end{proof}

\section{Proof of Theorem \ref{thm1}}
Throughout this section we assume that $\theta(n)=\sigma(n)+1$, so that $\mathcal{B}=\mathcal{P}$ and $\chi(n)$ is the 
characteristic function of the set of practical numbers. 
We need to estimate $\alpha = (1-e^{-\gamma}) c  = \lim_{N\to \infty} \alpha_N,$ where 
 \begin{equation}\label{aNdef}
 \alpha_N =   \sum_{n\le N} \frac{\chi(n)}{n}\left(\sum_{p\le \theta(n)}\frac{\log p}{p-1} - \log n\right)
\prod_{p\le \theta(n)} \left(1-\frac{1}{p}\right).
\end{equation}
We write $\alpha = \alpha_N + \beta_N$, with $\alpha_N$ being computable by \eqref{aNdef}. To estimate $\beta_N$, note that \eqref{etadef}
implies
\begin{equation}\label{bNdef} \beta_N = \sum_{n > N} \frac{\chi(n)}{n}\Bigl( \eta(\theta(n)) +\log(\theta(n)/n) -\gamma \Bigr)
\prod_{p\le \theta(n)} \left(1-\frac{1}{p}\right).
\end{equation}
Let
$$  \varepsilon_N = \sum_{n>N}  \frac{\chi(n)}{n}\prod_{p\le \theta(n)} \left(1-\frac{1}{p}\right) 
=1 - \sum_{n\le N}  \frac{\chi(n)}{n}\prod_{p\le \theta(n)} \left(1-\frac{1}{p}\right),$$
by Lemma \ref{lem1}.
The last equation allows us to calculate $\varepsilon_N$ on a computer.
The contribution from $-\gamma$ to \eqref{bNdef} is $-\gamma  \varepsilon_N $.
If $n$ is practical, then $n-1$ can be written as the sum of some proper divisors of $n$, so $\sigma(n)-n\ge n-1$ and 
$\theta(n)=\sigma(n)+1\ge 2n$. 
The contribution from $\eta(\theta(n))$ to \eqref{bNdef} in absolute value is therefore at most
$$ \sum_{n > N} \frac{\chi(n)}{n}| \eta(\theta(n))| 
\prod_{p\le \theta(n)} \left(1-\frac{1}{p}\right) \le E(2N) \varepsilon_N,$$
where
$$ E(x) := \sup_{y\ge x}|\eta(y)|.$$
It remains to estimate the contribution from $\log(\theta(n)/n)$ to \eqref{bNdef}, that is
$$ \sum_{n > N} \frac{\chi(n)}{n}\log\left(\frac{\sigma(n)+1}{n}\right)
\prod_{p\le \theta(n)} \left(1-\frac{1}{p}\right) = T_N +D_N,$$
where
$$ T_N :=  \sum_{n > N} \frac{\chi(n)}{n}\log(\sigma(n)/n)
\prod_{p\le \theta(n)} \left(1-\frac{1}{p}\right)$$
and
$$ 0< D_N =   \sum_{n > N} \frac{\chi(n)}{n}\log(1+1/\sigma(n))
\prod_{p\le \theta(n)} \left(1-\frac{1}{p}\right) < \frac{ \varepsilon_N}{2N},$$
since $0< \log(1+1/\sigma(n)) <  1/\sigma(n) \le 1/(2n-1) < 1/(2N)$. 
Combining these estimates, we have 
\begin{equation}\label{alphabounds}
 \alpha_N + T_N + \varepsilon_N (-\gamma  - E(2N)) < \alpha < \alpha_N + T_N + \varepsilon_N (-\gamma +1/(2N) + E(2N)).
\end{equation}
To compute $T_N$ we write $T_N=A - U_N$, where
$$ U_N =  \sum_{n \le N} \frac{\chi(n)}{n}\log(\sigma(n)/n)
\prod_{p\le \theta(n)} \left(1-\frac{1}{p}\right)$$
can be calculated on a computer and 
$$ A = \sum_{n \ge 1} \frac{\chi(n)}{n} \sum_{q\le \theta(n)} W_q \prod_{p\le \theta(n)} \left(1-\frac{1}{p}\right),$$
by Lemma \ref{lem3}.
We write
$$ A= \sum_{1 \le n \le N} + \sum_{n>N} =A_N +V_N,$$
say. For practical $n>N$ we have $\theta(n)>2N$, so  
$$  \sum_{q\le 2N} W_q \le  \sum_{q\le \theta(n)} W_q <  \sum_{q\ge 2} W_q  
<  \sum_{q\le 2N} W_q + \sum_{q>2N} \frac{1}{q(q-1)}<  \sum_{q\le 2N} W_q + \frac{1}{2N},$$
by Lemma \ref{lem3}.
Thus
\begin{equation}\label{Vbounds}
\varepsilon_N \sum_{q\le 2N} W_q < V_N < \varepsilon_N \left(\sum_{q\le 2N} W_q + \frac{1}{2N} \right).
\end{equation}
To calculate $W_q$ efficiently, we write
\begin{equation*}
\begin{split}
  W_q &=\frac{-\log(1-1/q)}{q}+(q-1)\sum_{h\ge 1} \frac{\log(1-1/q^{h+1})}{q^{h+1}} \\
& =\frac{-\log(1-1/q)}{q}-(q-1)\sum_{h\ge 1}\frac{1}{q^{h+1}} \sum_{j\ge 1} \frac{1}{j}\left(\frac{1}{q^{h+1}}\right)^j \\
& =\frac{-\log(1-1/q)}{q}-\sum_{j\ge 1} \frac{q-1}{j q^{j+1} (q^{j+1}-1)}\\
& =\frac{-\log(1-1/q)}{q}-\sum_{1\le j \le J} \frac{q-1}{j q^{j+1} (q^{j+1}-1)} - R_{q,J}=W_{q,J}-R_{q,J},\\
\end{split}
\end{equation*}
say, where 
$$0 < R_{q,J} < \frac{1}{J+1}\sum_{ j > J} \frac{q-1}{ q^{j+1} (q^{j+1}-1)} < \frac{1}{J}\sum_{ j > J} \frac{q-1}{ q^{2j+2}}
= \frac{1}{Jq^{2J+2} (q+1)}.
$$
Thus
$$  \sum_{q\ge 2} R_{q,J} 
< \sum_{q\ge 2}\frac{1}{Jq^{2J+2} (q+1)}
<\frac{1}{3J} \left(\frac{1}{2^{2J+2}}+\int_2^\infty x^{-2J-2}dx \right)
< \frac{1}{J2^{2J+3}}, $$
for $J\ge 2$.
Replacing $W_q$ by $W_{q,J} > W_q$ in \eqref{Vbounds} yields
\begin{equation}\label{Vbounds2}
\varepsilon_N \left( Y_{2N,J} - \frac{1}{J2^{2J+3}}\right) < V_N < \varepsilon_N \left( Y_{2N,J}+ \frac{1}{2N} \right),
\end{equation}
where
$$ Y_{2N,J}=\sum_{q\le 2N} W_{q,J} .$$
Replacing $W_q$ by $W_{q,J}$ in $A_N$ gives
\begin{equation}\label{ANbounds}
   A_{N,J}-\frac{1-\varepsilon_N}{J2^{2J+3}}< A_N < A_{N,J},
\end{equation}
where
$$A_{N,J}=\sum_{n \le N } \frac{\chi(n)}{n} \sum_{q\le \theta(n)} W_{q,J} \prod_{p\le \theta(n)} \left(1-\frac{1}{p}\right).$$
Since $T_N=A-U_N=A_N+V_N-U_N$, combining \eqref{Vbounds2} and \eqref{ANbounds} with \eqref{alphabounds} yields
the lower bound 
$$
\alpha > \alpha_N + A_{N,J}-U_N -\frac{1}{J2^{2J+3}}  + \varepsilon_N \Bigl( Y_{2N,J}  -\gamma  - E(2N)\Bigr) 
$$
and the upper bound
$$
\alpha <  \alpha_N + A_{N,J}-U_N  + \varepsilon_N \left(Y_{2N,J}+ \frac{1}{N} -\gamma  + E(2N)\right).
$$
We let $J=13$ and $N=2^{31}$, so that $E(2N)\le M_{32}=2.174 \times 10^{-5}$ by Lemma \ref{pslem}.
Dividing both bounds for $\alpha$ by $ 1-e^{-\gamma}$, we get 
\begin{equation}\label{cint}
 1.33607322< c < 1.33607654.
\end{equation}

\section{Discussion}

Without precomputing the products and sums over primes, calculating $\alpha_N$, $\varepsilon_N$, $A_{N,J}$ and $U_N$ would take
$N^{2+o(1)}$ steps. To avoid this, we make a table with the practical numbers $n \le N$ in the first column and 
the values of $\theta(n)=\sigma(n)+1$ in the second column. We then sort the table according to $\theta(n)$. 
Next, we compute the three quantities 
$$   \prod_{p\le \theta(n)} \left( 1-\frac{1}{p}\right), \quad \sum_{p\le \theta(n)}\frac{\log p}{p-1}, \quad \sum_{q\le \theta(n)} W_{q,J},$$
recursively, for increasing values of $\theta(n)$, 
and store these values in columns 3 through 5 of our table. Finally, we sort the entire table according to $n$ in the first column.
Creating this table takes $N^{1+o(1)}$ steps and $N^{1+o(1)}$ bytes of memory. Calculating  $\alpha_N$, $\varepsilon_N$, $A_{N,J}$ and $U_N$,
with the use of this table, requires $N^{1+o(1)}$ steps.

With $J=13$ and $N=2^{31}$, the calculations took just over thirteen hours on a computer with sixteen gigabytes of RAM. 
Increasing $N$ further would require more memory, because of the large table. 

Note that the gap between the upper and lower bound for $\alpha$ is 
$$2\varepsilon_N E(2N) + \frac{\varepsilon_N}{N} + \frac{1}{J 2^{2J+3}} = 2\varepsilon_N E(2N) +O(1/N),$$
if $J 2^{2J+3}\ge N$. It follows from \eqref{Pasymp} that $\varepsilon_N \sim c e^{-\gamma}/\log N$. 
After dividing by $ 1-e^{-\gamma}$, the gap for $c$ is asymptotic to 
\begin{equation}\label{gapasymp}
  \frac{2c e^{-\gamma} E(2N)}{(1-e^{-\gamma})\log N}. 
\end{equation}

The width of the interval in \eqref{cint} is $3.32\times 10^{-6}$. If we increase $N$
from $2^{31}$ to $2^{36}$, Lemma \ref{pslem} would allow us to replace $M_{32}=2.174\times 10^{-5}$ by $M_{37}=0.438 \times 10^{-5}$ as an upper bound for $E(2N)$. With \eqref{gapasymp}, we expect the width of the interval for $c$ to be about $0.6\times 10^{-6}$ when $N=2^{36}$, while our algorithm would require about $2^5=32$ times as much memory and computing time compared to $N=2^{31}$.

Assuming the Riemann hypothesis, we have \cite[Lemma 13]{CFAE}
\begin{equation}\label{RH}
E(x) = \sup_{y\ge x}|\eta(y)| \le \frac{\log^2 x}{7\sqrt{x}} \qquad (x\ge 25),
\end{equation}
so that the gap for $c$ is $O( \log N / \sqrt{N})$ by \eqref{gapasymp}. However, the upper bounds listed in Table \ref{table1}, which are best possible, 
are significantly better than \eqref{RH}. 

Without the Riemann hypothesis, combining the best known error term in the prime number theorem with \eqref{deltadiff}, \eqref{deleta} and \eqref{gapasymp}, we find that the gap for $c$ is $O\left( \exp\left\{-(\log N)^{3/5-\varepsilon}\right\}\right)$, for every $\varepsilon >0$. 

\section*{Acknowledgments}
The author is grateful to the anonymous referee for several very helpful suggestions, 
to Jianlong Han for the use of his computer to complete the calculations,
and to Maurice Margenstern and Eric Saias for their comments after reading an earlier version of this paper.


\begin{thebibliography}{99}
\bibitem{Axler}
C. Axler, New estimates for some functions defined over primes, \textit{Integers} \textbf{18} (2018), \# A52, 21pp.

\bibitem{But}
J. B\"{u}the, An analytic method for bounding $\psi(x)$. \textit{Math. Comp.} \textbf{87} (2018), no. 312, 1991--2009. 

\bibitem{Dus}
P. Dusart, Explicit estimates of some functions over primes, \textit{Ramanujan J.} \textbf{45} (2018), no. 1, 227--251. 

\bibitem{Mar}
M. Margenstern, Les nombres pratiques: th\'{e}orie, observations et conjectures,
\textit{J. Number Theory} \textbf{37} (1991), 1--36.

\bibitem{RS}
J. B. Rosser and L. Schoenfeld, Approximate formulas for some functions of prime numbers, \textit{Illinois J. Math.} \textbf{6} (1962) 64--94. 

\bibitem{Sier}
W. Sierpinski, Sur une propri\'et\'e des nombres naturels,
\textit{Ann. Mat. Pura Appl. (4)} \textbf{39} (1955), 69--74.

\bibitem{Sri}
A. K. Srinivasan, 
Practical numbers, \textit{Current Sci.} \textbf{17} (1948), 179--180.

\bibitem{Stew}
B. M. Stewart, Sums of distinct divisors, \textit{Amer. J. Math.} \textbf{76} (1954), 779--785. 

\bibitem{PDD}
A. Weingartner, Practical numbers and the distribution of divisors, \textit{Q. J. Math.} \textbf{66} (2015), no. 2, 743--758. 

\bibitem{SPA}
A. Weingartner, A sieve problem and its application, 
\textit{Mathematika} \textbf{63} (2017), no. 1, 213--229. 

\bibitem{CFAE}
A. Weingartner, On the constant factor in several related asymptotic estimates, \textit{Math. Comp.} \textbf{88} (2019), no. 318, 1883--1902.

\end{thebibliography}
\end{document}